\documentclass[11pt]{article}
\usepackage{epsfig,graphics,graphicx}
\usepackage{amssymb}
\usepackage{amsmath}
\usepackage{extarrows}
\usepackage[utf8]{inputenc}
\usepackage[english]{babel}
\usepackage{geometry}
\usepackage{amsthm}
\usepackage{scalerel}[2014/03/10]
\usepackage[usestackEOL]{stackengine}
\usepackage{esint}

\newtheorem{theo}{Theorem}

\newtheorem{coro}[theo]{Corollary}

\def\R{\mathbb R}
\def\D{\mathbb D}
\def\C{\mathbb C}
\def\K{\mathbb K}
\def\v{\mathbf v}
\def\div{\operatorname{div}}

\title{\sc{Rotation bounds \\for H\"older continuous homeomorphisms\\ with integrable distortion}}
\author{\sc{A. Clop, L. Hitruhin, B. Sengupta}}
\date{}

\begin{document}

\maketitle

\abstract{We obtain sharp rotation bounds for the subclass of homeomorphisms $f:\C\to\C$ of finite distortion which have distortion function in $L^p_{loc}$, $p>1$, and for which a H\"older continuous inverse  is available. The interest in this class is partially motivated by examples arising from fluid mechanics. Our rotation bounds hereby presented improve the existing ones, for which the H\"older continuity is not assumed. We also present examples proving sharpness.} 

\section{Introduction}

We say that $f:\C\to\C$ is a mapping of finite distortion if it belongs to the Sobolev space $f\in W^{1,1}_{loc}(\C; \C)$, its jacobian determinant $\det (Df)=J(\cdot,f)$ is locally integrable, and there exists a measurable function $\K(\cdot, f):\C\to[1,+\infty]$ such that
$$|Df(z)|^2\leq \K(z,f)\cdot J(z,f)$$
at almost every point $z\in\C$. Above, $|Df(z)|$ stands for the operator norm of the differential matrix $Df(z)$ of $f$ at the point $z$. When $\K(\cdot, f)\in L^\infty$ then $f$ is said to be $K$-quasiregular, with $K=\|\K(\cdot, f)\|_\infty$ (or $K$-quasiconformal, if bijective). In the same way quasiregular maps arose as a generalization of holomorphic functions, mappings of finite distortion arose as a generalization of quasiregular maps partially motivated by questions in nonlinear elasticity. The authors address the interested reader to the monograph \cite{IM} for quasiregular maps in the plane, and to \cite{AIM} for a background on mappings of finite distortion.\\
\\
Recently there has been a growing interest in understanding the rotational properties of planar homeomorphisms, see \cite{AIPS, B, H1, H2, H3, H4}. Special attention has been devoted to the spiraling rate of these maps. More precisely, given a homeomorphism $f:\C\to \C$ normalized by $f(0)=0$ and $f(1)=1$, one is interested in the growth of $|\arg(f(r))|$ as $r\to 0$. This growth represents the number of times that the image $f([r,1])$ winds around the origin as $r\to 0$. This quantity has been proven to admit several speeds of growth which depend on the class of maps under study. In this way, it was proven in \cite{AIPS} that if $f$ is $K$-quasiconformal then
\begin{equation}\label{qc}
|\arg(f(r))|\leq \frac12\left(K-\frac1K\right)\,\log\left(\frac1r\right)+c_K,\hspace{1cm}\text{for all }0<r<1.
\end{equation}
In contrast, if the maps under study are homeomorphisms of finite distortion, the situation changes and the order of growth depends on the integrability of the distortion function. Namely, the second named author discovered in \cite{H2} that if $e^{\K(\cdot, f)}\in L^p_{loc}$ for some $p>0$ then  
$$|\arg(f(z))|\leq \frac{c}{p}\,\log^2\left(\frac1{|z|}\right),\hspace{1cm}\text{for small enough }|z|,$$
and moreover this is sharp up to the value of the constant $c>0$. In other words, the transition between boundedness and exponential integrability of $\K(\cdot, f)$ results in a larger power of the logarithmic term. Further optimal results were obtained later on in \cite{H3}, in the case of integrable distortion, that is, when $\K(\cdot, f)\in L^p_{loc}$ for some $p>1$,
\begin{equation}\label{Lpdist}
|\arg(f(z))|\leq\frac{c}{|z|^\frac2p},\hspace{1cm}\text{for small enough }|z|
\end{equation}
or even when $\K(\cdot, f)\in L^1_{loc}$,
\begin{equation}\label{L1dist}
\lim_{|z|\to 0} |z|^2\,|\arg(f(z))|=0.
\end{equation}
The moral here is that more spiraling is allowed at the cost of relaxing the integrability properties of $\K(\cdot, f)$. As explained in \cite{AIPS, H2, H3}, the local rotational properties go hand in hand with the local stretching behavior. Especially important for the argument are the estimates for the modulus of continuity of the inverse map. \\
\\
It turns out mappings of finite distortion also have a role in fluid mechanics. To be precise, let us think of the planar incompressible Euler system of equations in vorticity form,
\begin{equation}\label{euler}
\begin{cases}
\frac{d}{dt}\omega+ (\v \cdot\nabla )\omega = 0\\
\div(\v )=0\\
\omega(0,\cdot)=\omega_0.\end{cases}
\end{equation}
Here $\omega=\omega(t,z):[0,T]\times\C\to\C$ is the unknown, $\omega_0\in L^\infty(\C;\C)$ is given, and $\v $ is the velocity field. The Biot-Savart law,
$$\v =\frac{i}{2\pi \bar{z}}\ast \omega$$
makes more precise the relation between $\v $ and $\omega$. As it is well known, Yudovich \cite{Y} proved existence and uniqueness of a solution $\omega\in L^\infty([0,T]; L^\infty(\C;\C))$ for any given $\omega_0$. In particular, the corresponding velocity field $\v $ belongs to the Zygmund class, and therefore the classical Cauchy-Lipschitz theory guarantees for the ODE
$$
\begin{cases}
\frac{d}{dt} X(t,z)= \v (t, X(t,z))\\
X(0,z)=z
\end{cases}
$$
both existence and uniqueness of a flow map $X:[0,T]\times\C\to\C$. It was proven in \cite{CJ} that, for small enough $t>0$, each of the flow homeomorphisms $X_t=X(t,\cdot):\C\to\C$ is indeed a mapping of finite distortion. Moreover, for each small value $t>0$ there is a number $p(t)>1$ such that the distortion function $\K(\cdot, X_t)$ belongs to $L^p_{loc}$ whenever $p<p(t)$. \\
\\
As mappings with $L^p$ distortion, the mappings $X_t$ are a bit special because both $X_t$ and $X_{t}^{-1}$ are H\"older continuous, as shown in \cite{W}, with a H\"older exponent that decays exponentially in time. This is not true in general, and mappings of $L^p$ distortion need not have a H\"older continuous inverse, as shown in \cite{KT}. Therefore, it is a question of interest to find out if the H\"older nature of the inverse map results in better rotation bounds. Indeed, even though the bounds obtained in \cite{H3} can be applied to $X_t$,  the H\"older continuous nature of $X_t^{-1}$ provides a significant improvement to \eqref{Lpdist}. We describe this improvement in our main Theorem. 
 
\begin{theo}\label{main}
Let $f:\C\to\C$ be a homeomorphism of finite distortion such that $f(0)=0$ and $f(1)=1$, and assume that $\K(\cdot, f)\in L^p_{loc}$ for some $p> 1$. Suppose also that
$$|f(x)-f(y)|\geq C\,|x-y|^\alpha,\hspace{1cm}\text{if }|x-y|\text{ is small,}$$
for some $\alpha>1$. Then 
$$|\arg(f(z))|\leq C\,\sqrt{\alpha}\, |z|^{-\frac1p}\,\log^\frac12\left(\frac1{|z|}\right)$$
whenever $|z|$ is small enough. 
\end{theo}

%\noindent
%As it already happened in \cite{H3} the case $p=1$ gives even  better bound. 

%\begin{theo}\label{mainp=1}
%Let $f:\C\to\C$ be a homeomorphism of finite distortion normalized by the conditions $f(0)=0$ and $f(1)=1$, and assume that $\K(\cdot, f)\in L^1_{loc}$. Suppose also that
%$$|f(x)-f(y)|\geq C\,|x-y|^\alpha,\hspace{1cm}\text{if }|x-y|\text{ is small,}$$
%for some $\alpha>0$. Then 
%$$\limsup_{|z|\to 0} |\arg(f(z))|\,\frac{|z|}{\log^\frac12\left(\frac1{|z|}\right)}=0.$$
%\end{theo}

\noindent
In contrast with \eqref{Lpdist} and \eqref{L1dist}, the existence of a H\"older continuous inverse allows the power term exponent to be halved, although then the logarithmic term needs to be included.  As an application, we can estimate the spiraling rate of $X_t$.\\
\\
As an application, we can estimate the spiraling rate of $X_t$ for small times. The rotational behavior of $X_t$ is nowadays studied a lot. For instance, in the case of $\omega_0$ being \emph{close} to the characteristic function of the unit disk, the article \cite{CHJ} provides bounds for the winding number of most of the trajectories $\{X_{t}(z)\}_{t>0}$ as $t\to\infty$. Here, instead, we do not evaluate the rotational behavior at large times, but look instead at spiraling bounds \emph{in the space variable} for a fixed and small enough time.

\begin{coro}\label{coroeuler}
Given $\omega_0\in L^\infty(\C;\C)$, let $\v$ be the velocity field of Yudovich's solution to \eqref{euler}, and let $X_t$ be its flow. Then there exists a constant $C>0$ such that
$$
\left|\arg\left(\frac{X_t(z)-X_t(0)}{X_t(1)-X_t(0)}\right)\right|\leq 
C\, \log^\frac12\left(\frac1{|z|}\right)\,|z|^{-t\|\omega_0\|_\infty}\,\exp\left(Ct\|\omega_0\|_\infty\right)\,
$$
if both $|z|$ and $t>0$ are small enough. 
\end{coro}

\noindent
In particular, if $z=\frac1n$, $n=1,2,\dots$ and we fix a time $t_0$ small enough, then the curve $X_{t_0}([\frac1n, 1])$ cannot wind around $X_{t_0}(0)$ more than a multiple of 
$$n^{t_0\|\omega_0\|_\infty}\,(\log n)^\frac12 \,e^{Ct_0\|\omega_0\|_\infty}$$
times. Towards the optimality of Theorem 1, we can show the following.

\begin{theo}\label{main2}
Given an increasing, onto homeomorphism $h:[0,+\infty)\to[0,+\infty)$, and a real number $p>1$, there exists a homeomorphism $\bar{f}:\C\to\C$ with the folllowing properties:
\begin{itemize}
\item $\bar{f}$ is a mapping of finite distortion, with $\K(\cdot, \bar{f})\in L^p_{loc}$.
\item $\bar{f}(0)=0$, $\bar{f}(1)=1$.
\item If $\alpha>\frac{3p}{p-1}$, then $|\bar{f}(x)-\bar{f}(y)|\geq C|x-y|^\alpha$ whenever $|x-y|<1$.
\item There exists a decreasing sequence $\{r_n\}$, with $r_n\to 0+$ as $n\to \infty$, for which
$$|\arg(\bar{f}(r_n))|\geq r_n^{-\frac1p}\,\log^\frac12\left(\frac1{r_n}\right)\,h(r_n).$$
\end{itemize}
\end{theo}
\noindent
Since $h$ can be chosen to approach $0$ at any speed, Theorem \ref{main2} shows that the order provided in Theorem \ref{main} is sharp.  \\
\\
The paper is structured as follows. In Section 2 we give the basic preliminaries. In Section 3 we prove  Theorem \ref{main} and Corollary \ref{coroeuler}. Finally, we prove Theorem \ref{main2} in Section 4. \\
\\
\textbf{Acknowledgements}.  A.C. and B.S. are partially supported by projects MTM2016-81703-ERC, MTM2016-75390 (Spanish Government) and 2017SGR395 (Catalan Government). L.H. was partially supported by ICMAT Severo Ochoa project SEV-2015-0554 grant MTM2017-85934-C3-2-P, the ERC grant 307179-GFTIPFD, the ERC grant 834728 Quamap, by the Finnish Academy of Science project 13316965 and by a grant from The Emil Aaltonen Foundation.

\section{Preliminaries}

A mapping $f:\C\to\C$ is said to be H\"older continuous, or simply H\"older from above, if there exist constants $C>0$, $d>0$ and $\alpha>0$ such that for any two points $x,y\in\C$ and $\alpha\in\R^{+}\setminus\{0\}$ with $|x-y|<d$ one has
$$|f(x)-f(y)|\leq C|x-y|^\alpha$$
Similarly, we say $f$ is H\"older from below if there are constants $C, \beta>0$ such that for any two points $x,y\in\C$ and $\beta\in\R^{+}\setminus\{0\}$ with $|x-y|<d$ one has
$$|f(x)-f(y)|\geq\bar{C}|x-y|^\beta $$
A mapping $f:\C\to\C$ is called bi-H\"older if it is both H\"older from above and from below.\\
\\ 
Let $f: \mathbb{C} \to \mathbb{C}$ be a mapping of finite distortion and fix a point $z_0 \in \mathbb{C}$. In order to study the pointwise rotation of $f$ at the point $z_0$, one usually fixes an argument $\theta\in[0,2\pi)$, and then looks at how the quantity
 $$\arg (f(z_0+te^{i\theta})-f(z_0))$$ 
changes as the parameter $t$ goes from 1 to a small $r$. This can also be understood as the winding of the path $f\left( [z_0+re^{i\theta}, z_0+e^{i\theta}] \right)$ around the point $f(z_0)$. As we are interested in the maximal pointwise spiraling, we need to normalize and then retain the maximum over all directions $\theta$,
	\begin{equation}\label{spiraling}
	\sup_{\theta \in [0,2\pi)} |\arg (f(z_0+re^{i\theta})-f(z_0)) - \arg (f(z_0+e^{i\theta})-f(z_0))  |.
	\end{equation}
Then, the maximal pointwise rotation is precisely the behavior of the above quantity \eqref{spiraling} when $r\to 0^+$. In this way, we say that the map $f$ \emph{spirals at the point $z_0$ with a rate $g$}, where $g: [0, \infty) \to [0, \infty)$ is a decreasing continuous function, if 
	\begin{equation}\label{Spiral rate}
	\limsup_{r\to 0^+} \frac{\sup_{\theta \in [0,2\pi)} |\arg (f(z_0+re^{i\theta})-f(z_0)) - \arg (f(z_0+e^{i\theta})-f(z_0))  |}{g(r)} = C
	\end{equation}
for some constant $0<C<\infty$. Finding maximal pointwise rotation for a given class of mappings equals finding the maximal spiraling rate for this class. Note that in \eqref{Spiral rate} we must use limit superior as the limit itself might not exist. Furthermore, for a given mapping $f$ there might be many different sequences $r_n \to 0$ along which it has profoundly different rotational behaviour. \\
\\
Our proof of Theorem \ref{main} relies heavily on the modulus of path families. We give here the main definitions, and address the interested reader to \cite{V} for a closer look at the topic. The image of a line segment $I$ under a continuous mapping is called a \emph{path}, and we denote by $\Gamma$  a family of paths. Given a path family $\Gamma$, we say that a Borel measurable function $\rho$ is \emph{admissible for $\Gamma$} if any rectifiable $\gamma \in \Gamma$ satisfies
	\begin{equation*}
	\int_{\gamma} \rho(z)dz \geq 1.
	\end{equation*}
The \emph{modulus of the path family $\Gamma$} is defined by 
	\begin{equation*}
	M(\Gamma)= \inf_{\rho \text{ admissible}} \int_{\mathbb{C}} \rho^{2}(z)\,dA(z),
	\end{equation*}
where $dA(z)$ denotes the Lebesgue measure on $\C$. As an intuitive rule, the modulus is big if the family $\Gamma$ has \emph{lots} of short paths, and it is small if the paths are long and there are \emph{not many} of them. \newline \newline 
	We will also need a weighted version of the modulus. Any measurable, locally integrable function $\omega: \mathbb{C} \to [0, \infty)$ will be called \emph{a weight function}. In our case, $\omega$ will always be the distortion function $\K(\cdot, f)$ of some map $f$. Then, we define the weighted modulus $M_{\omega}(\Gamma)$ by 
	\begin{equation*}
	M_{\omega}(\Gamma)= \inf_{\rho \text{ admissible}} \int_{\mathbb{C}} \rho^{2}(z)\, \omega(z)\,dA(z).
	\end{equation*}
	Finally, we need the modulus inequality
	\begin{equation}\label{Modeq}
	M(f(\Gamma)) \leq M_{\K(\cdot, f)}(\Gamma)
	\end{equation}
	which holds for any mapping $f$ of finite distortion for which the distortion $\K(\cdot, f)$ is locally integrable, proven by the second named author in \cite{H3}. 

\section{Spiraling bounds}
We will write Theorem \ref{main} in the following, clearly equivalent, form.
\begin{theo}
Let $f$ be a homeomorphism of finite distortion with distortion $\K(\cdot, f) \in L^{p}(\C)$, $p>1$, normalized by $f(0)=0$ and $f(1)=1$. Assume that it satisfies the following condition,
$$\aligned
|f(x)-f(y)|\geq C|x-y|^{\alpha }
\endaligned$$ whenever $|x-y|$ is small.  Then the winding number $n(z_0)$ of the image of the line segment $\left[z_0, \frac{z_0}{|z_0|}\right]$ around the image of the origin is bounded from above by
$$\aligned
n(z_0)&\leq C\sqrt{\alpha}\, |z_0|^{-\frac{1}{p}}\,\log^\frac12\left(\frac{1}{|z_0|}\right) \\
\endaligned$$
\end{theo}

\begin{proof}
We would like to prove this theorem using the modulus inequality for homeomorphisms of finite distortion \eqref{Modeq} following the presentation in \cite{H3}. At first, we would like to estimate the modulus term $M_{\K(\cdot, f)}(\Gamma)$ from above. To this end,
let us choose an arbitrary point $z_0\in\C\setminus\{0\}$ such that $|z_0|<1$. Without loss of generality, we may assume that $z_0$ lies on the positive side of the real axis. Next, let us fix the line segments $E=[z_0,1]$ and $F=(-\infty,0]$, and $\Gamma$ be the family of paths connecting a point in $E$ to a point in $F$. Also, let us fix balls $B_j=B(2^{j}z_0,2^{j}z_0)$, $j\in\{0,1,...,n\}$ and $n$ is the smallest integer such that $2^{n}z_0\geq1$. Define
$$\aligned
\rho_{0}(z)=
\begin{cases}
\frac{2}{r(B_0)}\hspace{0.2cm}\mbox{if}\hspace{0.2cm}z\in B_0\\
\frac{2}{r(B_1)}\hspace{0.2cm}\mbox{if}\hspace{0.2cm}z\in B_1\setminus B_0\\
\vdots\\
\frac{2}{r(B_{n})}\hspace{0.2cm}\mbox{if}\hspace{0.2cm}z\in B_n\setminus B_{n-1}\\
0\hspace{0.6cm}\mbox{otherwise}
\end{cases}
\endaligned$$ Note that any $z\in E$ belongs to some ball $\frac{1}{2}B_j$ and that $\rho_{0}(z)\geq\frac{2}{r(B_j)}$, whenever $z\in B_j$. This implies, since $B_j\cap F=\emptyset$ for every $j$, that $\rho_0(z)$ is admissible with respect to $\Gamma$. Hence we can estimate the modulus from above by 
$$\aligned
M_{\K(\cdot, f)}(\Gamma)&=\inf_{\rho \text{ admissible}}\int_{\C}\K(\cdot, f)\rho^{2}(z)\,dA(z)\\
                 &\leq\int_{\C}\K(\cdot, f)\rho_{0}^{2}(z)\,dA(z)\\
                 &\leq\|\K(\cdot, f)\|_{L^{p}\left(B(0,4)\right)}\left(\int_{B(0,4)}\rho_{0}^\frac{2p}{p-1}(z)\,dA(z)\right)^\frac{p-1}{p}\\
                 &\leq c_{f,p}\left(\int_{B(0,4)}\rho_{0}^\frac{2p}{p-1}(z)\,dA(z)\right)^\frac{p-1}{p}\\
\endaligned$$ Let us now estimate the integral term by using the definition of $\rho_0$.
$$\aligned
\int_{B(0,4)}\rho_{0}^{\frac{2p}{p-1}}(z)\,dA(z)&\leq\sum_{j=0}^{n}\int_{B_j}\left(\frac{2}{r(B_j)}\right)^\frac{2p}{p-1}\,dA(z)\\
                                           &=\sum_{j=0}^{n}|B_j|\left(\frac{2}{r(B_j)}\right)^\frac{2p}{p-1}\\
                                           &=c_{p}\sum_{j=0}^{n}\frac{\left(r(B_j)\right)^2}{\left(r(B_j)\right)^\frac{2p}{p-1}}\\
                                           &=c_p\sum_{j=0}^{n}\frac{1}{z_{0}^{\frac{2}{p-1}}}\frac{1}{2^{\frac{2j}{p-1}}}\\
                                           &=c_{p}z_{0}^{-\frac{2}{p-1}}\sum_{j=0}^{n}\frac{1}{2^{\frac{2j}{p-1}}}\\
\endaligned$$ The series $\sum_{j=0}^{n}\frac{1}{2^{\frac{2j}{p-1}}}$ converges to a constant depending on p for any fixed $p>1$. Therefore, 
$$\aligned
M_{\K(\cdot, f)}(\Gamma)\leq c_{f,p}z_{0}^{-\frac{2}{p}}
\endaligned$$
Next, we would like to estimate the modulus term $M\left(f(\Gamma)\right)$ from below. Let us start with the definition of $M\left(f(\Gamma)\right)$ in polar coordinates
$$\aligned
M\left(f(\Gamma)\right)&=\inf_{\rho \text{ admissible}}\int_{\C}\rho^{2}(z)\,dA(z)\\
                       &=\inf_{\rho \text{ admissible}}\int_{0}^{2\pi}\int_{0}^{\infty}\rho^{2}(r,\theta)r\hspace{0.1cm}drd\theta\\
\endaligned$$ and provide a lower bound for 
$$\aligned
\int_{0}^{\infty}\rho^{2}(r,\theta)rdr
\endaligned$$ for an arbitrary direction $\theta\in[0,2\pi)$ and an arbitrary admissible $\rho$.
To this end, we fix a direction $\theta$ and consider the half line $L_\theta$ starting from the origin in the direction $\theta$.
Let us choose points $z_0 \leq t_2< t_0\leq 1$ so that the image set $f(E)$ winds once around the origin when $z$ moves from a point $t_0$ to a point $t_2$ along $E$ and $f(t_0)\in L_\theta$. Since the mapping $f$ is a homeomorphism and the path $f(F)$ contains both the origin and points with big modulus, as $F$ is unbounded, the path $f(F)$ must intersect the line segment $(f(t_2),f(t_0))$ at least once, say at a point $f(t_1)$, with $t_1\in F$. We can choose $t_1$ and $t_0$ such that there are no points from
the paths $f(E)$ and $f(F)$ in the line segment $(f(t_1), f(t_0))$, which thus belongs to the path family $f(\Gamma)$.
Since the path $f(E)$ cycles around the origin $n(z_0)=\left \lfloor{\frac{|\arg\left(f(z_0)\right)-\arg\left(f(1)\right)|}{2\pi}}\right \rfloor
$ times we can find at least
$$\aligned
n(z_0)=\left \lfloor{\frac{|\arg\left(f(z_0)\right)-\arg\left(f(1)\right)|}{2\pi}}\right \rfloor-1
\endaligned$$ 
such disjoint line segments belonging to the path family $f(\Gamma)$. Note that $n(z_0)$ does not depend on the direction $\theta$. Since we are interested in extremal rotation, it can be assumed that $f(E)$ winds around the origin at least once, which makes it clear that $n(z_0)$ is non-negative. Now, the $n(z_0)$ disjoint line segments can be written in the form $\left(x_{j}e^{i\theta},y_{j}e^{i\theta}\right)\subset L_\theta$, where $j\in\{1,2,...,n(z_0)\}$ and $x_j$,$y_j$ are positive real numbers satisfying 
$$\aligned
0<r_f\leq x_1<y_1<...<x_{n(z_0)}<y_{n(z_0)}\leq c_f
\endaligned$$ where $c_f=\sup_{z\in E}|f(z)|$ and $r_f=\min_{z\in E}|f(z)|$. Here, neither $c_f$ nor $r_f$ depends on $\theta$ or $z_0$. So, one could write
$$\aligned
\int_{0}^{\infty}\rho^{2}(r,\theta)rdr\geq\sum_{j=1}^{n(z_0)}\int_{x_j}^{y_j}\rho^{2}(r,\theta)rdr.
\endaligned$$
Next, let us consider the H\"older inequality with the functions $f(r)=\rho\sqrt{r}$ and $g(r)=\frac{1}{\sqrt{r}}$, which after squaring both sides gives
$$\aligned
\int_{x_j}^{y_j}\rho^{2}(r,\theta)rdr\geq\left(\int_{x_j}^{y_j}\rho(r,\theta)dr\right)^{2}\left(\int_{x_j}^{y_j}\frac{1}{r}dr\right)^{-1}\geq\frac{1}{\log\left(\frac{y_j}{x_j}\right)}.
\endaligned$$ The last inequality holds true as $\rho$ is admissible with respect to $f(\Gamma)$ where the line segments $(x_{j}e^{i\theta},y_{j}e^{i\theta})$ belong to the path family $f(\Gamma)$.
Therefore,
$$\aligned
\int_{0}^{\infty}\rho^{2}(r,\theta)rdr\geq\sum_{j=1}^{n(z_0)}\frac{1}{\log\left(\frac{y_j}{x_j}\right)}.
\endaligned$$ It is quite clear from the definition of $c_f$ that
$$\aligned
\sum_{j=1}^{n(z_0)}\frac{1}{\log\left(\frac{y_j}{x_j}\right)}\geq\sum_{j=1}^{n(z_0)-1}\frac{1}{\log\left(\frac{x_{j+1}}{x_j}\right)}+\frac{1}{\log\left(\frac{c_f}{x_{n(z_0)}}\right)}.
\endaligned$$
Next, let us consider the AM-HM inequality. For every positive real number $a_j$,
$$\aligned
\sum_{j=1}^{n}a_j\geq\frac{n^2}{\sum_{j=1}^{n}\frac{1}{a_j}}.
\endaligned$$
At this point, we would like to use AM-HM with the precise choices 
$$a_j=\frac{1}{\log\left(\frac{x_{j+1}}{x_j}\right)}\hspace{.5cm}\text{ if }j\in\{1,2,...,n(z_0)-1\},\text{  and }\hspace{1cm}a_{n(z_0)}=\frac{1}{\log\left(\frac{c_f}{x_{n(z_0)}}\right)},$$ 
which give 
$$\aligned
\sum_{j=1}^{n(z_0)}\frac{1}{\log\left(\frac{y_j}{x_j}\right)}\geq\frac{n^{2}(z_0)}{\log\left(\frac{c_f}{x_1}\right)}\geq\frac{n^{2}(z_0)}{\log\left(\frac{c_f}{r_f}\right)}.
\endaligned$$
Therefore,
$$\int_{0}^{\infty}\rho^{2}(r,\theta)rdr\geq\frac{n^{2}(z_0)}{\log\left(\frac{c_f}{r_f}\right)}.$$
The constant $c_f$ can be defined as $\max_{z\in\D}|f(z)|$, which is finite and does not depend on either $\theta$ or $z_0$, and thus it is irrelevant at the limit $z_0\to0$. On the other hand,  the constant $r_f$ must be estimated using the H\"older modulus of continuity assumption on our mapping $f$, that is
$$|f(z_0)|\geq C|z_0|^{\alpha}$$ 
for sufficiently small $z_0$. This combined with the estimate above gives that 
$$M\left(f(\Gamma)\right)\geq\frac{n^{2}(z_0)}{C\alpha\log\left(\frac{1}{|z_0|}\right)}$$ 
Now, using the modulus inequality we get
$$ \frac{n^{2}(z_0)}{C\alpha\log\left(\frac{1}{|z_0|}\right)}\leq c_{f,p}\left(\frac{1}{|z_0|}\right)^\frac{2}{p}$$ 
which implies the desired estimate. 
\end{proof}

%\noindent 
%We now prove Theorem \ref{mainp=1}

%\begin{proof}[Proof of Theorem \ref{mainp=1}] 
%Next, let us deal with the case $p=1$. Let us choose 
%$$\aligned
%\rho_{0}(z)=
%\begin{cases}
%\frac{1}{z_0}\hspace{0.4cm}\mbox{if}\hspace{0.2cm}dist(z,E)<z_0\\
%0\hspace{0.6cm}otherwise\\
%\end{cases}
%\endaligned$$Note that $\rho_0$ is admissible with respect to the path family $\Gamma$. Therefore,
%$$\aligned
%M_{\K(\cdot, f)}(\Gamma)&\leq\int_{\C}\K(z, f)\rho_{0}^{2}(z)\,dA(z)\\
%                 &=\frac{1}{z_{0}^{2}}\int_{\{z:dist(z,E)<z_0\}}\K(z, f)\,dA(z)\\
%                 &=\frac{c_{f}(z_0)}{z_{0}^{2}}
%\endaligned$$Now, we know that $\K(\cdot, f) \in L^{1}_{loc}(\C)$ and $\left|\{z:dist(z,E)<z_0\}\right|\to0$. So, $c_{f}(z_0)\to0$ as $z_0\to0$.
%We already have the following inequality
%$$M\left(f(\Gamma)\right)\geq\frac{n^{2}(z_0)}{C\alpha\log\left(\frac{1}{|z_0|}\right)}$$
%Next, use the modulus inequality to get
%$$ \frac{n^{2}(z_0)}{C\alpha\log\left(\frac{1}{|z_0|}\right)}\leq\frac{c_{f}(z_0)}{z_{0}^{2}}$$
%that implies
%$$ 
%n(z_0)\leq\sqrt{\alpha\,c_{f}(z_0)}\,\frac{\log^\frac12\left(\frac{1}{|z_0|}\right)}{z_0}
%$$
%which gives the desired bound. 
%\end{proof}

\begin{proof}[Proof of Corollary \ref{coroeuler}]
Corollary \ref{coroeuler} follows immediately after noting that one can take $f=X_t$ in Theorem \ref{main}. Indeed,  from \cite[Corollary 3]{CJ}  we know that $X_t$ belongs to $W^{1,p}$ for any $p<\frac{2}{t\|\omega_0\|_\infty}$, provided that $0<t<\frac{2}{\|\omega_0\|_\infty}$. Since $J(\cdot, X_t)=1$ due to the incompressibility, it then follows $X_t$ is a homeomorphism with finite distortion, and moreover $\K(\cdot, X_t)\in L^{p}_{\text{loc}}$ for $p<\frac{1}{t\|\omega_0\|_\infty}$. Especially, if $t$ is so small that $0<t<\frac{1}{\|\omega_0\|_\infty}$ then one may take $p>1$. Also, we recall from \cite{W} (see also \cite{BC}) that $X_t^{-1}$ is $\alpha$-H\"older continuous with some exponent $\alpha\geq e^{-ct\,\|\omega_0\|_\infty}$ for some $c>0$. Hence, Theorem \ref{main} applies to $f=X_t$ and the claim follows. 
\end{proof}

\section{Optimality of spiraling}

%We would like to prove the optimality of our result by constructing some non-trivial examples. To this end, it will be useful to recall from \cite{AIPS} that the maps $z\mapsto\frac{z}{|z|}|z|^{q+i\alpha}$ are $K$-quasiconformal, with quasiconformality constant $K$ determined by
%\begin{equation}\label{aips}
%\left|\left(q+i\alpha \right)-\frac{1}{2}\left(K+\frac{1}{K}\right)\right|=\frac{1}{2}\left(K-\frac{1}{K}\right)
%\end{equation} 
%In particular, if $q=1$ then $z\mapsto\frac{z}{|z|}|z|^{1+i\alpha}$ is bilipschitz, as its jacobian is constantly $1$. Moreover, also when $q=1$ one gets from \eqref{aips} that $\alpha^2=(\sqrt{K}-1/\sqrt{K})^2$, which means that $\alpha\simeq\sqrt{K}$ provided that $K$ is large enough. We will repeatedly use the above maps in our building blocks. \\\\
\noindent
We will get  Theorem \ref{main2} in two steps. In the first step, we will construct a map which \emph{only rotates}. This map will already give us the optimal result (in the power scale). In the second step, we will strengthen this up with a second map, that \emph{both rotates and stretches}. This second map is going to be the optimal one.\\
\\
Given an arbitrary annulus $A=B(0,R)\setminus B(0,r)$ we define the corresponding rotation map as 
$$\aligned
\phi_{A}(z)=
\begin{cases}
z&|z|>R\\
z \,e^{ i\alpha\log\left|\frac{z}{R}\right| }&r\leq |z|\leq R\\
z\,e^{i\alpha \,\log\frac{r}{R} }&|z|<r
\end{cases}
\endaligned$$ 
Here $0<r<R$, and $\alpha \in\R$. One must note that $\phi_{A}:\C\to\C$ is bilipschitz (i.e. both $\phi_A$ and its inverse are Lipschitz), hence quasiconformal (its quasiconformality constant depends only on $\alpha$), and moreover  it is conformal outside the annulus $A$. Note also that $\phi_A$ leaves fixed all circles centered at $0$, since $|\phi_A(te^{i\theta})|=t$ for each $t>0$ and $\theta\in\R$. Finally, a direct calculation shows for the jacobian determinant that  $J(z, \phi_A)=1$ for each $z$. \\
\\
Next, we fix a sequence $\{r_n\}$ such that $0<r_{n+1}<\frac{r_n}{2e}$ and $r_1<\frac{1}{e}$. Also, let $R_n=e r_n$. These assumptions make sure that $2r_{n+1}<R_{n+1}<\frac{r_n}2$. Let us now construct disjoint annuli $A_n=B(0,R_n)\setminus B(0,r_n)$, and set $\{f_n\}_n$ to be a sequence of maps, constructed in an iterative way as follows. For $n=1$, we set
$$\aligned
f_{1}(z)=\phi_{A_1}(z)=
\begin{cases}
z&|z|>R_1\\
z \,e^{ i\alpha_1\,\log \frac{|z| }{R_1}}&r_1\leq|z|\leq R_1\\
z\,e^{-i\alpha_1 }&|z|<r_1\\
\end{cases}
\endaligned$$
where $\alpha_1\in\R$, $\alpha_1\geq 1$, is to be determined later. We then define $f_n$ for $n\geq2$ as 
$$\aligned
f_{n}(z)=\phi_{f_{n-1}(A_n)}\circ f_{n-1}(z)
\endaligned$$
again for some values $\alpha_n\in\R$, $\alpha_n\geq 1$, to be determined later. Clearly, each $f_n:\C\to\C$ is quasiconformal, and conformal outside the annuli $A_i$, $i=1,\dots, n$. It is also clear that $f_n(z)=f_{n-1}(z)$ on the unbounded component of $\C\setminus f_{n-1}(A_n)$ (i.e. outside of $B(0, R_n)$). This proves that the sequence $f_n$ is uniformly Cauchy and hence it converges to a map $f$, that is,
$$\aligned
f=\lim_{n\to\infty}f_n
\endaligned$$
which is again a homeomorphism by construction. Now, since $f_n$ is quasiconformal for every $n$ and $f_n(z)=f_{n-1}(z)$ everywhere except inside the ball $B(0,R_n)$, where $R_n\to0$ as $n\to\infty$, the limit map $f$ is absolutely continuous on almost every line parallel to the coordinate axes and differentiable almost everywhere. \\
\\
It is helpful to note that each $f_n$ leaves fixed all circles centered at the origin, so in particular we have $f_n(A_j)=A_j$ for each $j$, and therefore $\phi_{f_{n-1}(A_n)}=\phi_{A_n}$. Direct calculation shows that
$$
|D\phi_{A_n}(z)|=|\partial\phi_{A_n}(z)|+|\overline\partial\phi_{A_n}(z)|=
\begin{cases}
1&|z|>R_n\\
\frac{|2+i\alpha_n|+|\alpha_n|}{2}&r_n\leq |z|\leq R_n\\
1&|z|< r_n\end{cases}
$$
which allows us to estimate that
$$\aligned
|\partial f(z)|+|\bar\partial f(z)|\leq 2\alpha_n\hspace{1cm}\text{whenever }z\in A_n,
\endaligned$$
and $|Df(z)|\leq1$ otherwise. Therefore, in order to have $Df(z)\in L^{1}_{loc}(\C)$ it suffices that
\begin{equation}\label{dfinl1}
\sum_n \alpha_n\,r_n^2<+\infty.
\end{equation}
This, together with the absolute continuity, guarantees $f\in W^{1,1}_{loc}(\C)$. Also, since $f$ is a homeomorphism, we have that $J_{f}(z)\in L^{1}_{loc}(\C)$, and in fact $J(z, f)=1$ at almost every $z\in\C$. Therefore, $f$ is a homeomorphism of finite distortion, with distortion function
$$\aligned
\K(z,f)=\frac{|Df(z)|^2}{J(z, f)}\leq
\begin{cases}
4\alpha_n^2&z\in A_n,\\
1&\text{otherwise.}
\end{cases}
\endaligned$$
Especially, in order to have $\K(\cdot, f)\in L^p_{loc}$, it suffices to ensure the convergence of the series
\begin{equation}\label{kinlp}
\sum_{n=1}^{\infty}|A_n|(4\alpha_n^2)^{p}\simeq \sum_{n=1}^\infty \alpha_n^{2p}\,r_n^2
\end{equation}
which can be done by choosing $\alpha_n$ properly. Note that if \eqref{kinlp} holds, then also \eqref{dfinl1} holds, because our choice of $\alpha_n$ will guarantee $\alpha_n\geq 1$. The last restriction to choose our $\alpha_n$ comes from rotational behavior of $f$. It is clear from the above construction that $f(0)=0$, $f(1)=1$ and
$$
\left|\arg\left(f(r_n)\right)\right|\geq\left|\arg\left(\left(\frac{1}{e}\right)^{1+i\alpha_n}\right)\right|=\alpha_n
$$
for every $r_n$. Since we want our map to be optimal for Theorem \ref{main}, we may be tempted to choose $\alpha_n= r_n^{-1/p}\,\log^{1/2}(1/r_n)$. Unfortunately such a choice does not meet the requirement \eqref{kinlp}. The same problem occurs if we simply choose $\alpha_n=r_n^{-1/p}$. So we choose
$$\alpha_n=h(r_n)\,r_n^{-1/p}.$$
Here $h:[0,\infty)\to[0,\infty)$ is any monotonically decreasing gauge function such that $\lim_{r\to 0^+}h(r)=0$.  With this choice, \eqref{kinlp} is fulfilled if the series
$$\sum_{n=1}^\infty h(r_n)^{2p}<+\infty.$$
But this can always be done by simply reducing the already chosen values of $r_n$, for instance if $h(r_n)<\frac{1}{n^{1/2}}$. Note that this does not provide full optimality for Theorem \ref{main}, but it already gives the right order (in the power scale).\\
\\
We now show that $f$ is H\"older continuous with exponent $1-\frac1p$. For this, let us recall that our map $f$ is a limit of iterates of logarithmic spiral maps inside the annuli $A_n=B(0,R_n)\setminus B(0,r_n)$. In particular, as shown in \cite{AIPS}, if $\gamma\in\R$ then the basic logarithmic spiral map $g(z)=z|z|^{i\alpha}=ze^{i\gamma\log|z|}$ is $L$-bilipschitz, for a constant $L$ such that $|\gamma|=L-\frac{1}{L}$. When $|\gamma|$ is large, $L$ is large as well and so one roughly has $|\gamma|\simeq L$. Since our $f_n$ behaves on the annulus $A_n$ as a spiral map with $|\gamma|=\alpha_n$, we deduce the bilipschitz constant of $f_n$ on $A_n$ is 
$$L \simeq|\gamma| =\alpha_n= h(r_n)\,r_n^{-1/p}.$$ 
Let us now start the proof. To this end, let us consider two arbitrary points $x$ and $y$ in $\D\setminus\{0\}$. We first consider the case where $x, y\in A_n$. In this case, $f(x)=f_n(x)$ and $f(y)=f_n(y)$. Since $r_n>C|x-y|$, we have
$$\aligned
|f(x)-f(y)|=|f_n(x)-f_n(y)| &\lesssim h(r_n)\,r_n^{-1/p}|x-y|\\
           &\leq h(r_n) \left(\frac{C}{|x-y|}\right)^{\frac{1}{p}}|x-y|\\
           &\leq C|x-y|^{1-\frac{1}{p}}
\endaligned$$ 
where we have used the bilipschitz nature of $f_n$ on $A_n$.\\
\\
We now assume that $x, y\in D_n=B(0,r_n)\setminus B(0,R_{n+1})$. On that set $f$ is of the form $ze^{i\beta}$, where $\beta\in\R\setminus\{0\}$, which is clearly an isometry.  \\
\\
Next, we take $x\in A_n$ and $y\in D_n$. In particular, $|x|\geq|y|$. Then let $w$ be any point on the outer boundary of $D_n$ joining $x$ and $y$. We have
$$\aligned
|f(x)-f(y)|&\leq|f(x)-f(w)|+|f(w)-f(y)|\\
           &\leq C|x-w|^{1-\frac{1}{p}}+|w-y|\\
           &\leq2C|x-y|^{1-\frac{1}{p}}
\endaligned$$
The same happens if $x\in D_{n-1}$ and $y\in A_n$. \\
\\
So it just remains to see what happens when $x\in A_n=B(0,R_n)\setminus B(0,r_n)$ and $y\in B(0,R_{n+1})$. Let $L$ be the line joining $x$ and $y$. We divide it into three parts, viz., $L_1$, $L_2$ and $L_3$. $L_1$ connects $x$ to a point $a$ on the inner boundary of $A_n$, so that
$$|f(x)-f(a)|=|f_n(x)-f_n(a)|\leq C|x-a|^{1-\frac1p}$$
Next, $L_2$ connects $a$ to $b$, which is the closest point to $y$ where the line $L$ crosses the inner boundary of $D_n$. From  $2R_{n+1}<r_n<\frac{R_n}2$ we get that $|f(a)|>2|f(b)|$. Also, since $a,b\in D_n$ and $f$ is an isometry there, we get
$$|f(b)-f(y)|\leq2|f(b)|\leq2|f(a)-f(b)|=2|a-b|$$ 
Summarizing
$$\aligned
|f(x)-f(y)|&\leq|f(x)-f(a)|+|f(a)-f(b)|+|f(b)-f(y)|\\ 
           &\leq C|x-a|^{1-\frac{1}{p}}+|a-b|+2|a-b|\\
           &\leq C|x-y|^{1-\frac{1}{p}}
\endaligned$$
The case $x\in D_n$ and $y\in B(0,r_{n+1})$ can be proved in a similar manner. Therefore, we have covered all the possible cases. Since the set $\D\setminus\{0\}$ is partitioned by separated annuli $A_n$ and $D_n$, it is clear that we have proved that $f$ is H\"older continuous with exponent $1-\frac1p$. At this point, it is worth noting that this regularity could also be proven by means of the Sobolev embedding. Indeed, we proved above that $\K(\cdot, f)\in L^p_{loc}$, and also that the Jacobian determinant is constantly $1$. This together implies that $Df\in L^{2p}_{loc}$. \\
\\
Now we show that also $f^{-1}$ is H\"older continuous. Indeed, let us recall that $f$ is the limit of iterates of logarithmic spiral maps inside the annuli and conformal outside. Now, $f^{-1}$ can be constructed using the same building blocks as $f$ itself, just changing the sign of $\alpha_n$ at each step. This is possible because the inverse of a logarithmic spiral map is the same spiral map, just the direction of rotation is opposite of the original map. Since it is clear that the direction of rotation does not play any role in the proof of H\"older continuity of $f$, this implies that $f^{-1}$ is also H\"older from above. Thus $f$ is H\"older from below as well. \\
\\
As we said before, the above example approaches the borderline stated in Theorem \ref{main}, but it does not attain full optimality yet. To this end, we have to modify it by adding to our building blocks a stretching factor. This is done by replacing, at each iterate, the logarithmic spiral map $z|z|^{i\alpha}=ze^{i\alpha\log|z|}$ by a complex power $z|z|^{q+i\alpha}=z|z|^q\,e^{i\alpha\log|z|}$. We now proceed with the details.\\
\\
So, similarly as in the previous construction, we fix a rapidly decreasing sequence $\{r_n\}$ such that $r_{n+1}<\frac{r_n}{2e}$ and $r_1<\frac{1}{e}$. Also, let $R_n=e r_n$. Given an arbitrary annulus $A=B(0,R)\setminus B(0,r)$ we define the corresponding radial stretching combined with rotation map as follows:
\begin{equation}\label{buildingblock}
\aligned
\phi_{A}(z)=
\begin{cases}
z&|z|>R\\
z\,\left|\frac{z}{R}\right|^{q-1}\,e^{i\alpha\log\frac{|z|}{R}}&r\leq|z|\leq R\\
z\left(\frac{r}{R}\right)^{q-1}\,e^{i\alpha\log\frac{r}{R}}&|z|<r\
\end{cases}
\endaligned
\end{equation} 
Note that this time we will have $q\geq 1$. Direct calculation shows that
$$
|\partial\phi_A(z)|+|\overline\partial\phi_A(z)|
=\begin{cases}
1&|z|>R\\
R^{1-q}|z|^{q-1}\frac{|q+1+i\alpha|+|q-1+i\alpha|}{2}&r \leq |z|\leq R \\
R^{1-q}r^{q-1}&|z|<r
\end{cases}$$
and also that
$$
J(z,\phi_A)
=\begin{cases}
1&|z|>R\\
q\,\left(\frac{|z|}{R}\right)^{2(q-1)} &r \leq |z|\leq R \\
\left(\frac{r}{R}\right)^{2(q-1)} &|z|<r
\end{cases}$$
whence
$$
\K(z,\phi_A)
=\begin{cases}
1&|z|>R\\
\frac{(|q+1+i\alpha|+|q-1+i\alpha|)^2}{4q}&r \leq |z|\leq R \\
1&|z|<r
\end{cases}$$
In particular, if $2\leq q+1<\alpha$ then one may estimate $\|\K(\cdot, \phi_A)\|_\infty\leq \frac{4\alpha^2}{q}$. Next, let us construct the sequence of maps $f_n$ in an iterative way as follows. For $n=1$, we set
$$\aligned
f_{1}(z)=\phi_{A_1}(z)=
\begin{cases}
z&|z|<R_1\\
z\,\left|\frac{z}{R_1}\right|^{q_1-1}\,e^{i\alpha_1\log\frac{|z|}{R_1}}&r_1\leq |z|\leq R_1\\
z\left(\frac{1}{e}\right)^{q_1-1}\,e^{-i\alpha_1}&|z|<r_1\\
\end{cases}
\endaligned$$
where $q_1$ and $\alpha_1$ are to be determined later. Next, assuming we have $f_1,\dots, f_{n-1}$, we define $f_n$ for $n\geq2$ as:
$$\aligned
f_{n}(z)=\phi_{f_{n-1}(A_n)}\circ f_{n-1}(z)
\endaligned$$
Note that $\phi_{f_{n-1}(A_n)}$ is determined by the inner and outer radii of $\phi_{f_{n-1}(A_n)}$ (which are already available since $f_1,\dots, f_{n-1}$ are known) as well as for  the parameters $q_n$ and $\alpha_n$, which will be determined later. Clearly, each $f_n:\C\to\C$ is quasiconformal, and conformal outside the annuli $A_i$, $i\in\{1,...,n\}$. Moreover, one can easily show that 
$$ 
\K(\cdot , f_n) = \prod_{j=1}^{n}\K(\cdot, f_{n-j} \circ \phi_{f_{n-j}(A_{n-j+1})})\\
 =\prod_{j=1}^{n}\K( \cdot,\phi_{ A_{n-j+1} }) $$
so that $\K(z, f_n)\leq C\frac{\alpha_j^2}{q_j}$ whenever $z\in A_j$, $j=1\dots n$ while $\K(\cdot,f_n)=1$ otherwise. In a similar way, we can use that $|D\phi_A(z)|\leq C\alpha $ when $z\in A$ (and $|D\phi_A(z)|\leq 1$ at all other points) to obtain that $|Df_n|\leq C \alpha_j$ on $A_j$, $j=1\dots n$, and $|Df_n|\leq 1$ otherwise. \\
\\
By construction, we have $f_n(z)=f_{n-1}(z)$ whenever $z\notin B(0, R_n)$. Thus $\{f_n \}_n$ converges uniformly to a map $\bar{f}(z)$, that is,
$$\aligned
\bar{f}=\lim_{n\to\infty}f_n
\endaligned$$
which is again a homeomorphism by construction. A similar argument to the one before shows that $\bar{f}$ is absolutely continuous on almost every line parallel to the coordinate axis. For almost every fixed $z_0\neq 0$ there is a neighbourhood of $z_0$ such that the sequence $\{f_n(z)\}_n$ remains constant for $n$ very large and $z$ in that neighbourhood. Therefore the same happens to the sequences $Df_n(z)$, $J(z, f_n)$ and $\K(z, f_n)$, and so their limits are precisely $D\bar{f}(z)$, $J(z, \bar{f})$ and $\K(z, \bar{f})$. Especially, in order to have $D\bar{f}\in L^1_{loc}$ it suffices that
\begin{equation}\label{dfinl1general}
\sum_{n=1}^\infty |A_n|\,\alpha_n<+\infty
\end{equation}
In case this holds true, then $\bar{f}$ is a homeomorphism in $W^{1,1}_{loc}$, and as a consequence its jacobian determinant $J(\cdot, \bar{f})\in L^1_{loc}$. Moreover, in order to have $\K(\cdot, \bar{f})\in L^p_{loc}$ one needs to require that
\begin{equation}\label{kfinlpgeneral}
\sum_{n=1}^\infty |A_n|\,\frac{\alpha_n^{2p}}{q_n^p}<\infty
\end{equation}
Again, as it was the case for $f$, \eqref{kfinlpgeneral} implies \eqref{dfinl1general} when $q_n^\frac{p}{2p-1}< \alpha_n$ and so our parameters $\alpha_n$ and $q_n$ need to be chosen according to \eqref{kfinlpgeneral} as well as the purpose of $\bar{f}$ to be optimal for Theorem \ref{main}. For this, again as before, we have $\bar{f}(0)=0$, $\bar{f}(1)=1$ and
$$ 
\left|\arg\left(\bar{f}(r_n)\right)\right|\geq\left|\arg\left(\left(\frac{1}{e}\right)^{q_n+i\alpha_n}\right)\right|=|\alpha_n|
$$
which motivates us to choose 
$$
\alpha_n= h(r_n)\, \left(\log\frac{1}{r_n}\right)^{1/2}\,r_n^{-\frac{1}{p}}\hspace{2cm}q_n=\log\frac1{r_n},
$$
where $h$ is any gauge function such that $h(r)\to0$ as $r\to0$ and the condition $q_n^\frac{p}{2p-1}< \alpha_n$ is satisfied. Indeed, with these choices \eqref{kfinlpgeneral} becomes
$$\sum_nh(r_n)^{2p}<\infty$$
which, as before, may always be granted by choosing smaller $r_n$, if needed. Having \eqref{kfinlpgeneral} fulfilled, our map $\bar{f}$ is a mapping of finite distortion with $\K(\cdot, \bar{f})\in L^p_{loc}$. Also, the resulting map $\bar{f}$ attains the optimal rotational behavior stated at Theorem \ref{main} modulo the gauge function $h$ which can be chosen to converge to $0$ as slowly as desired. \\
\\
Therefore, Theorem \ref{main2} will be proven if we are able to show that $\bar{f}$ is H\"older from below. Furthermore, we also show that $\bar{f}$  is H\"older from above, highlighting regularity of our mappings.   \\
\\
 To do this, we first observe that the composition of $z\mapsto ze^{i\alpha\log|z|}$ followed by $z\mapsto z|z|^{q-1}$ is precisely $z\mapsto z|z|^{q-1}e^{i\alpha\log|z|}$. This observation suggests us to decompose $\bar{f}=g\circ f$, where $f$ is essentially the first example in this section (with different choice of $\alpha_n$) and $g$ is constructed by building blocks \eqref{buildingblock} with $\alpha=0$ at each step. Morally, $f$ leaves fixed all circles centered at $0$, and only rotates the annuli $A_n$, while $g$ conveniently stretches each $A_n$.   \\ \\
For any $p>1$, the bi-H\"older nature of $f$ has already been proven when $\alpha_n=h(r_n)\,r_n^{-1/p}$. Hence we can directly use the same proof there after we estimate $$ h(r_n)\, \left(\log\frac{1}{r_n}\right)^{1/2}\,r_n^{-\frac{1}{p}} \leq h(r_n)\,r_n^{-1/(p- \epsilon)}$$ for all small $r_n$ and $\epsilon = (p-1)/2$. Therefore, it only remains to show that $g$ is bi-H\"older as well. To this end, we first show that $g$ is H\"older from above using the fundamental theorem of calculus.\\
\\
Let $x,y\in B(0,1)$ be given. Without loss of generality let us assume that $|y|\geq|x|$ and let $w$ be the point for which $|w|=|x|$ and $\arg(w)=\arg(y)$. Now  
\begin{equation} \label{H-A}
|g(x)-g(y)| \leq |g(w)-g(x)|+ |g(y)-g(w)|,
\end{equation}
and we will show that both of these are H\"older.  First, since $g$ maps circles centered at the origin radially to similar circles  with equal or smaller radius (as $q_n\geq 1$)  it is clear that 
\begin{equation*}
|g(x)-g(w)| \leq |x-w| \leq |x-y|.
\end{equation*}
Let us then concentrate of the second part. First we note, that we can without loss of generality assume that $y$ and $w$ are real numbers as $g$ is a radial mapping. From our construction we see that the line segments $[r_n, R_n]$, $(R_{n+1},r_n)$ and $(R_1,1]$ partition the line segment $(0,1]$. Furthermore, from  \eqref{buildingblock} it is clear that the differential is bounded from above by $1$ in the segments $(R_{n+1},r_n)$ and $(R_1,1]$. On the other hand, in segments  $[r_n, R_n]$ we can estimate 
\begin{equation*}
|g'(t)| \leq \log\left( \frac{1}{r_n} \right)  \leq \frac{C}{\sqrt{t}} 
\end{equation*}
for any $t \in [r_n, R_n]$ with fixed $C$ that does not depend on $n$ or $t$. This is so because of our choice of $q_n$. Combining these two estimates we have 
\begin{equation*}
|g'(t)| \leq \frac{C}{\sqrt{t}} 
\end{equation*}
for any $t \in (0,1)$. Thus we can use fundamental theorem of calculus to estimate 
$$\aligned
|g(y)-g(w)|&=\int_{w}^{y} |g'(t)|dt\\
&\leq\int_{w}^{y} \frac{C}{\sqrt{t}}  dt\\
&=2C\left(\sqrt{y} - \sqrt{w}\right)\\
&\leq 2C \sqrt{y-w}.
\endaligned$$ 
This proves that also the second part in \eqref{H-A} is H\"older, and thus we obtain 
\begin{equation*}
|g(y)-g(x)| \leq |g(y)-g(w)|+ |g(w)-g(x)| \leq C\sqrt{|y-w|} + \sqrt{|x-w|} \leq 2C \sqrt{|x-y|},
\end{equation*}
which shows $g$ is H\"older from above. \\ \\
Let us next prove that $g$ is H\"older from below. To this end, given any two points $x,y\in B(0,1)$ we again without loss of generality assume that $|y|\geq|x|$ and let $w$ be the point for which $|w|=|x|$ and $\arg(w)=\arg(y)$. Now, as $g$ is a radial homeomorphism, it follows that
$$\aligned
|g(x)-g(y)|\geq\max\{|g(x)-g(w)|,|g(y)-g(w)|\}
\endaligned$$
Moreover, 
$$\aligned
\max\{|x-w|,|y-w|\}\geq\frac{1}{2}|x-y|
\endaligned$$
Therefore, it is enough to show that both $|g(x)-g(w)|$ and $|g(y)-g(w)|$ satisfy H\"older bounds from below. Note that if $x=0$ then clearly $w=0$ and we have only the radial part $|g(y)-g(w)|$. \\ \\ 
Let us first check the term $|g(x)-g(w)|$. Since $g$ maps radially circles centered at the origin to similar circles we see that $|g(x)-g(w)|$ gets contracted the same amount as the modulus $|g(w)|$ is contracted under $g$. Now we must consider two possibilities, either $x,w \in A_n$ or $x,w \in D_n$ for some $n$. Let us first assume $x,w\in A_n=B(0,R_n)\setminus B(0,r_n)$ for some $n$. Here we impose an additional assumption that
\begin{equation}\label{Additional assumption Lauri}
r_n<\left(\frac{1}{e}\right)^{q_{n-1}+q_{n-2}+...+q_{1}-(n-1)},
\end{equation}
which we can do as  the radii  $r_n$ can be assumed to decrease as fast as we want. Then we can estimate
$$\aligned
|g(x)|&=\left(\frac{1}{e}\right)^{q_{n-1}+q_{n-2}+...+q_1-(n-1)}\cdot |x|\,\left(\frac{|x|}{R_n}\right)^{q_n-1}\\
      &\geq r_n \cdot |x|\,\left(\frac{|x|}{R_n}\right)^{q_n-1}\\
      &\geq r_n \cdot |x|\,\left(\frac1e\right)^{q_n-1} =e\cdot r_{n}^{2}\cdot|x|
\endaligned$$ for any $x\in A_n$. Therefore,
$$\aligned
|g(x)-g(w)|\geq e\cdot r_{n}^{2}\cdot|x-w|\\
           \geq C\cdot|x-w|^3
\endaligned$$
since $|x-w|<C\cdot r_n$ for some fixed constant $C>0$ when $x,w \in A_n$. \\
\\
Next, let $x,w\in D_n=B(0,r_n)\setminus B(0,R_{n+1})$ for some $n$.  Using \eqref{Additional assumption Lauri} we get
$$\aligned
|g(x)|&\geq c\left(\frac{1}{e}\right)^{q_{n-1}+q_{n-2}+...+q_{1}-(n-1)}\cdot r_n\cdot|x|\\
      &\geq c\cdot r_{n}^{2}\cdot|x|.
\endaligned$$
Thus we can use a similar argument as in the previous case to estimate
$$\aligned
|g(x)-g(w)|&\geq c\cdot r_{n}^{2}\cdot|x-w|\\
           &\geq c\cdot|x-w|^3\\
\endaligned$$since $|x-w|<c\cdot r_n$ for some fixed constant $c>0$ when $x,w\in D_n$.\\  
\\
Since the set $\D\setminus\{0\}$ is partitioned by separated annuli $A_n$ and $D_n$ we have thus proven that $|g(x)-g(w)|$ satisfies H\"older estimates from below. \\
\\
Finally, let us prove the H\"older estimates from below for the term $|g(y)-g(w)|$. As the mapping $g$ is radial, we can again  assume that $y$ and $w$ are real. We aim to use again the Fundamental Theorem of Calculus, and thus have to estimate the differential from below.  Using \eqref{Additional assumption Lauri}, as well as the fact that $q_n>1$, we can estimate for any real number $t\in [r_n,R_n]$ that
$$\aligned
g'(t)&=\left(\frac{1}{e}\right)^{q_{n-1}+q_{n-2}+...+q_1-(n-1)}\cdot   q_n\cdot \left(\frac{t}{R_n}\right)^{q_{n}-1}\\
     &\geq r_n\, q_n\cdot \left(\frac{r_n}{R_n}\right)^{q_{n}-1}\\    
      &= e\,q_n\,r_n^2 \\
      &\geq c\cdot t^2\,\log\frac1t.
\endaligned$$ 
Next, if $t\in [R_{n+1},r_n]$, we have
$$\aligned
g'(t) &=\left(\frac{1}{e}\right)^{q_{n-1}+q_{n-2}+...+q_{1}-(n-1)}\cdot\left(\frac{1}{e}\right)^{q_{n}-1}\\
&\geq e\cdot r_{n}^{2}\\
&\geq c\cdot t^2
\endaligned$$
Thus, as before, since $(0,1)$ is partitioned by the intervals $[r_n,R_n]$, $[R_{n+1},r_n]$ and $[R_1,1)$, we end up getting that   
$$\aligned
g'(t)\geq c\cdot t^2
\endaligned$$ for every $t\in(0,1)$. Now, we use the fundamental theorem of calculus to get
$$\aligned
|g(y)-g(w)|&=\int_{w}^{y}g'(t)dt\\
           &\geq\int_{w}^{y}c\cdot t^{2}dt\\
           &=C\left(y^{3}-w^{3}\right)\\
           &\geq C|y-w|^3
\endaligned$$ This proves that the second term is H\"older from below as well, which in turn proves that $g$ is H\"older from below. This finishes the proof of Theorem \ref{main2}.

 \noindent
 Albert Clop\\
 Department of Mathematics and Computer Science\\
 Universitat de Barcelona\\
 08007-Barcelona\\
 CATALONIA\\
 albert.clop@ub.edu\\
 \\
 \\
 Lauri Hitruhin\\
 Department of Mathematics and Systems Analysis\\
 Aalto University\\
 P.O. Box 11100 FI-00076 Aalto\\
 Helsinki, Finland\\
 lauri.hitruhin@aalto.fi\\
 \\
 \\
 Banhirup Sengupta\\
 Departament de Matem\`atiques\\
 Universitat Aut\`onoma de Barcelona\\
 08193-Bellaterra\\
 CATALONIA\\
 sengupta@mat.uab.cat\\
 
\end{document}